\documentclass[11pt]{article}

\usepackage{amssymb,amsmath,amsfonts,amsthm}
\usepackage{latexsym}
\usepackage{graphics}
\usepackage{indentfirst}
\usepackage{hyperref}
\usepackage[capitalise, noabbrev, nameinlink]{cleveref}
\usepackage{comment}
\usepackage[shortlabels]{enumitem}
\usepackage[english]{babel}
\usepackage{tikz}
\usetikzlibrary{fit,arrows.meta,backgrounds}
\usetikzlibrary{shapes,decorations,graphs,graphs.standard,quotes,backgrounds}
\usetikzlibrary{decorations.markings}
\usepackage{esdiff}
\usepackage{thmtools}
\usepackage{thm-restate}

\definecolor{azure(colorwheel)}{rgb}{0.0, 0.5, 1.0}
\definecolor{Pink}{RGB}{255, 105, 180}

\usetikzlibrary{arrows.meta}

\newtheorem*{thm*}{Theorem}
\newtheorem{thm}{Theorem}
\newtheorem{lem}[thm]{Lemma}

\newtheorem{ques}[thm]{Question}

\crefname{thm}{Theorem}{Theorems}
\crefname{lem}{Lemma}{Lemmas}

\newcommand{\N}{\mathbb{N}}

\allowdisplaybreaks

\begin{document}

\title{Enumeratively Chromatic-Choosable Theta Graphs}

\author{ 
Yanghong Chi  \thanks{Alabama School of Math and Science, Mobile, AL, USA (chiyanghong12138@gmail.com)}
\and
Seoju Lee   \thanks{Alabama School of Math and Science, Mobile, AL, USA (Seojuu.lee@gmail.com)}
\and
Fennec Morrissette \thanks{Department of Biomedical Sciences, University of South Alabama, Mobile, AL, USA (ajm2228@jagmail.southalabama.edu )}
\and
Jeffrey A. Mudrock \thanks{Department of Mathematics and Statistics, University of South Alabama, Mobile, AL, USA (mudrock@southalabama.edu)}
\and
Gavin Nguyen  \thanks{Alabama School of Math and Science, Mobile, AL, USA (Gavin.Nguyen@asms.net)}
\and
Benjamin Whatley
\thanks{Mitchell College of Business, University of South Alabama, Mobile, AL, USA (bvw2422@jagmail.southalabama.edu)}
}

\maketitle

\begin{abstract}
Chromatic choosability is a notion of fundamental importance in list coloring.  A graph $G$ is \emph{chromatic-choosable} when its chromatic number, $\chi(G)$, is equal to its list chromatic number $\chi_{\ell}(G)$.  In 1990, Kostochka and Sidorenko introduced the list color function of a graph $G$, denoted $P_{\ell}(G,m)$, which is the list analogue of the chromatic polynomial of $G$, $P(G,m)$.  A graph $G$ is said to be \emph{enumeratively chromatic-choosable} when $P_{\ell}(G,m)=P(G,m)$ for every $m \in \N$.  Theta graphs and their generalizations have played an important role in graph coloring problems over the years; for example, they appear in the characterization of chromatic-choosable graphs with chromatic number 2.  In this paper we characterize the enumeratively chromatic-choosable theta graphs.  Our proof utilizes ideas from DP-coloring (a.k.a. correspondence coloring), providing yet another example of how the more general setting of DP-coloring can be leveraged to attack a problem in list coloring.  
 
\medskip

\noindent {\bf Keywords.} list coloring, chromatic choosability, list color function, theta graphs, enumerative chromatic choosability

\noindent \textbf{Mathematics Subject Classification.} 05C15, 05C30

\end{abstract}

\section{Introduction}\label{intro}
In this paper, all graphs are nonempty, finite, simple graphs.  Generally speaking, we follow West~\cite{W01} for terminology and notation.  The set of natural numbers is $\N = \{1,2,3, \ldots \}$.  For $m \in \N$, we write $[m]$ for the set $\{1, \ldots, m \}$.  We also take $[0]$ to be the empty set. For a graph $G$, $V(G)$ and $E(G)$ denote the vertex set and the edge set of $G$, respectively.   When $u$ and $v$ are adjacent in $G$ we write $uv$ or $vu$ for the edge with endpoints $u$ and $v$.  For $k \in \N$ and $G=P_k$, when we say that \emph{the vertices of $G$ in order are $v_1,\ldots,v_k$} we mean that two vertices are adjacent in $G$ if and only if they appear consecutively in this ordering.  For $k\geq3$ and $G=C_k$, when we say that \emph{the vertices of $G$ in cyclic order are $v_1,\ldots,v_k$} we mean that two vertices are adjacent in $G$ if and only if they appear consecutively in this ordering or if they are $v_1$ and $v_k$.  We write $K_{m,n}$ for complete bipartite graphs with partite sets of size $m$ and $n$. 

\subsection{List Coloring and Chromatic Choosability} \label{basic}

For classical vertex coloring of graphs, we wish to color the vertices of a graph $G$ with up to $m$ colors from $[m]$ so that adjacent vertices receive different colors, a so-called \emph{proper $m$-coloring}.  The \emph{chromatic number} of a graph $G$, denoted $\chi(G)$, is the smallest $m$ such that $G$ has a proper $m$-coloring.  List coloring is a well-known variation on classical vertex coloring that was introduced independently by Vizing~\cite{V76} and Erd\H{o}s, Rubin, and Taylor~\cite{ET79} in the 1970s.  For list coloring, we associate a \emph{list assignment} $L$ with a graph $G$ such that each vertex $v \in V(G)$ is assigned a list of available colors $L(v)$ ($L$ is said to be a list assignment for $G$).  We say $G$ is \emph{$L$-colorable} if there is a proper coloring $f$ of $G$ such that $f(v) \in L(v)$ for each $v \in V(G)$ (we refer to $f$ as a \emph{proper $L$-coloring} of $G$).  A list assignment $L$ for $G$ is called a \emph{$k$-assignment} if $|L(v)|=k$ for each $v \in V(G)$.  Graph $G$ is said to be \emph{$k$-choosable} if it is $L$-colorable whenever $L$ is a $k$-assignment for $G$.  The \emph{list chromatic number} of a graph $G$, denoted $\chi_\ell(G)$, is the smallest $m$ for which $G$ is $m$-choosable.  It is easy to show that for any graph $G$, $\chi(G) \leq \chi_\ell(G)$.   

A graph $G$ is called \emph{chromatic-choosable} if $\chi(G) = \chi_{\ell}(G)$~\cite{O02}.   Determining whether a graph is chromatic-choosable is, in general, a challenging problem.  Perhaps the most famous conjecture involving list coloring is about chromatic choosability.  Indeed, the Edge List Coloring Conjecture states that every line graph of a loopless multigraph is chromatic-choosable (see~\cite{HC92}).  

\subsection{Enumeratively Chromatic-Choosable Graphs}

In 1912, Birkhoff~\cite{B12} introduced the notion of the chromatic polynomial in hopes of using it to make progress on the four color problem.  For $m \in \N$, the \emph{chromatic polynomial} of a graph $G$, denoted $P(G,m)$, is the number of proper $m$-colorings of $G$. It is easy to show that $P(G,m)$ is a polynomial in $m$ of degree $|V(G)|$ (see~\cite{B12}). For example, whenever $n \in \N$ it is well known that $P(K_n,m) = \prod_{i=0}^{n-1} (m-i)$, $P(C_n,m) = (m-1)^n + (-1)^n (m-1)$, and $P(T,m) = m(m-1)^{n-1}$ whenever $T$ is a tree on $n$ vertices, (see~\cite{B94, W01}).  

The notion of chromatic polynomial was extended to list coloring in the early 1990s by Kostochka and Sidorenko~\cite{AS90}.  If $L$ is a list assignment for a graph $G$, let $P(G,L)$ denote the number of proper $L$-colorings of $G$. For $m \in \N$, the \emph{list color function} of $G$, denoted $P_\ell(G,m)$, is the minimum value of $P(G,L)$ where the minimum is taken over all possible $m$-assignments $L$ for $G$.  Since an $m$-assignment could assign the same $m$ colors to every vertex in a graph, it is clear that $P_\ell(G,m) \leq P(G,m)$ for each $m \in \N$.  In general, the list color function can differ significantly from the chromatic polynomial for small values of $m$.  For example, for any $n \geq 2$, $P_{\ell}(K_{n,n^n},m) = 0$ and $P(K_{n,n^n},m) > 1$ whenever $m \in \{2, \ldots, n\}$.  On the other hand, in 2023 Dong and Zhang~\cite{DZ22}, improving on earlier results (see~\cite{D92, T09, WQ17}), showed that for any graph $G$, $P_\ell(G,m) = P(G,m)$ whenever $m \geq |E(G)|-1$.  

With this and chromatic choosability in mind, it is natural to try to determine which graphs $G$ satisfy $P_\ell(G,m) = P(G,m)$ for every $m \in \N$.  Such graphs are said to be \emph{enumeratively chromatic-choosable}.  Enumerative chromatic choosability was first formally defined in~\cite{KK23} even though it has been pursued since the introduction of the list color function in the early 1990s~\cite{AS90}.  One of the most important open questions on the list color function illustrates a key challenge to proving a graph is enumeratively chromatic-choosable. 

\begin{ques} [\cite{KN16}] \label{ques: threshold}
	For every graph $G$, if $P_{\ell}(G,m)=P(G,m)$ for some $m \geq \chi(G)$, does it follow that $P_{\ell}(G,m+1)=P(G,m+1)$?
\end{ques} 

Indeed, since we don't know the answer to \cref{ques: threshold}, showing that $P_{\ell}(G, \chi(G)) = P(G, \chi(G))$ doesn't imply that $G$ is enumeratively chromatic-choosable.  Relatively little is known about which graphs are enumeratively chromatic-choosable.  Our next theorem gives some examples of enumeratively chromatic-choosable graphs.  However, before we state the theorem we review some important terminology.  Recall that a graph is \emph{chordal} if it contains no induced cycles of length greater than 3.  For $l_1,l_2,l_3 \in \N$, \emph{theta graphs}, denoted $\Theta(l_1,l_2, l_3)$, consist of a pair of end vertices joined by $3$ internally disjoint paths of lengths $l_1$, $l_2$, and $l_3$. We say that the \emph{core} of a connected graph $G$ is the graph obtained from $G$ by successively deleting vertices of degree 1. 

\begin{thm}[\cite{AM25, BK23, HK23, KN16, AS90}] \label{thm: examples} 
The following statements hold.

\begin{enumerate}[label=(\roman*)]
    \item Chordal graphs are enumeratively chromatic-choosable.
    \item Cycles are enumeratively chromatic-choosable.
    \item If $1 \leq l_1 \leq l_2 \leq l_3$ and the parity of $l_1$ is different from both $l_2$ and $l_3$, then $\Theta(l_1, l_2, l_3)$ is enumeratively chromatic-choosable.
    \item Suppose $G$ is a connected graph with $\chi(G)=2$.  Then, $G$ is enumeratively chromatic-choosable if and only if the core of $G$ is a copy of: $K_1$, $C_{2k+2}$ for $k \in \N$, or $\Theta(2,2,2)$ (i.e., $K_{2,3}$). 
\end{enumerate}
\end{thm}

With \cref{thm: examples} in mind, the motivation for this paper was the following question.

\begin{ques} \label{ques: theta}
Which theta graphs are enumeratively chromatic-choosable?
\end{ques}

Aside from the progress already made toward \cref{ques: theta}, theta graphs and their generalizations have played an important role in graph coloring over the years.  Indeed generalized theta graphs have been widely studied (see e.g.,~\cite{BF16, BK23, CM15, ET79, HK23, LB16, LB19, MM21, SJ18}), and they are the main subject of two classical papers on the chromatic polynomial~\cite{BH01, S04} which include the celebrated result that the zeros of the chromatic polynomials of the generalized theta graphs are dense in the whole complex plane with the possible exception of a unit disc.

In this paper we use a generalization of list coloring called DP-coloring to answer \cref{ques: theta}. Specifically, we prove the following.

\begin{thm} \label{thm: characterize}
Suppose $G=\Theta(l_1,l_2,l_3)$ where $\min\{l_1,l_2,l_3\} = l_1$, $l_2 \geq 2$, and $l_3 \geq 2$.  Then, $G$ is not enumeratively chromatic-choosable if and only if $l_1,l_2,$ and $l_3$ all have the same parity and $\{l_1,l_2,l_3\} \neq \{2\}$.
\end{thm}

\section{Proof of Theorem~\ref{thm: characterize}}

Suppose $G=\Theta(l_1,l_2,l_3)$, where $\min\{l_1,l_2,l_3\}=l_1$, $l_2 \geq 2$, and $l_3 \geq 2$.  We immediately have that $G$ is enumeratively chromatic-choosable when the parity of $l_1$ is different from both $l_2$ and $l_3$ by Statement~(iii) of Theorem~\ref{thm: examples}.

Now, suppose $l_1,l_2,$ and $l_3$ all have the same parity. Then, $\chi(G) = 2$. The fact that $G$ is not enumeratively chromatic-choosable when $\{l_1,l_2,l_3\} \neq \{2\}$ and is enumeratively chromatic-choosable when $l_1=l_2=l_3=2$ follows from Statement~(iv) of Theorem~\ref{thm: examples}.

So, to complete the proof of Theorem~\ref{thm: characterize}, we may assume the parity of $l_1$ and $l_3$ is the same and the parity of $l_1$ and $l_2$ is different.  Then, we must prove that $G$ is enumeratively chromatic-choosable; that is, $P_{\ell}(G,m) = P(G,m)$ whenever $m \geq 3$ since $\chi(G)=3$.

The remainder of the proof of Theorem~\ref{thm: characterize} uses ideas from DP-coloring which we will now review.  The concept of DP-coloring was first put forward in 2015 by Dvo\v{r}\'{a}k and Postle under the name \emph{correspondence coloring} (see~\cite{DP15}).  Intuitively, DP-coloring generalizes list coloring by allowing the colors that are identified as the same to vary from edge to edge.  Formally, for a graph $G$, a \emph{DP-cover} (or simply \emph{cover}) of $G$ is an ordered pair $\mathcal{H}=(L,H)$, where $H$ is a graph and $L:V(G)\to 2^{V(H)}$ is a function satisfying the following conditions: 
    \begin{itemize}
		\item $\{L(v) : v \in V(G)\}$ is a partition of $V(H)$ into $|V(G)|$ parts, 

		\item for every pair of adjacent vertices $u$, $v\in V(G)$, the set of edges between $L(u)$ and $L(v)$, denoted $E_H\left(L(u),L(v)\right)$, is a matching (not necessarily perfect and possibly empty), and

		\item $\displaystyle E(H) = \bigcup_{uv \in E(G)} E_{H}(L(u),L(v)).$
    \end{itemize}
    
Suppose $\mathcal{H}=(L,H)$ is a cover of a graph $G$.  A \emph{transversal} of $\mathcal{H}$ is a set of vertices $T\subseteq V(H)$ containing exactly one vertex from each $L(v)$. A transversal $T$ is said to be \emph{independent} if $T$ is an independent set in $H$.  If $T$ is an independent transversal of $\mathcal{H}$, then $T$ is said to be a \emph{proper $\mathcal{H}$-coloring} of $G$, and $G$ is said to be \emph{$\mathcal{H}$-colorable}.  A \emph{$k$-fold cover} of $G$ is a cover $\mathcal{H}=(L,H)$ such that $|L(v)|=k$ for all $v\in V(G)$.  We say that a $k$-fold cover $\mathcal{H}=(L,H)$ of $G$ is \emph{full} when for each $uv \in E(G)$, $E_H(L(u),L(v))$ is a perfect matching.

Using an idea similar to the one Kostochka and Sidorenko used to introduce list color functions of graphs, the notion of chromatic polynomial was extended to the DP-coloring context in 2021~\cite{KM19}. Suppose $\mathcal{H} = (L,H)$ is a cover of a graph $G$, and let $P_{DP}(G, \mathcal{H})$ be the number of proper $\mathcal{H}$-colorings of $G$.  Then, the \emph{DP color function of $G$}, denoted $P_{DP}(G,m)$, is the minimum value of $P_{DP}(G, \mathcal{H})$ where the minimum is taken over all possible $m$-fold covers $\mathcal{H}$ of $G$.  

Now, suppose that $L$ is an $m$-assignment for $G$.  The \emph{cover of $G$ corresponding to $L$}, denoted $\mathcal{H}_L = (\Lambda_L,H_L)$, is the cover of $G$ defined as follows.  For each $v \in V(G)$, $\Lambda_L(v) = \{(v,c) : c \in L(v) \}$, and $H_L$ is the graph with vertex set $\bigcup_{v \in V(G)} \Lambda_L(v)$ and edges created so that for any $(u,c_1),(v,c_2) \in V(H_L)$, $(u,c_1)(v,c_2) \in E(H_L)$ if and only if $uv \in E(G)$ and $c_1 = c_2$.  Notice that if $\mathcal{C}$ is the set of proper $L$-colorings of $G$ and $\mathcal{T}$ is the set of proper $\mathcal{H}_L$-colorings of $G$, then the function $h: \mathcal{C} \rightarrow \mathcal{T}$ given by $h(f) = \{(v,f(v)) : v \in V(G) \}$ is a bijection.  So, for any graph $G$ and $m \in \N$,
\[P_{DP}(G, m) \leq P_\ell(G,m) \leq P(G,m).\]

We now turn our attention back to theta graphs.  Importantly, formulas for the chromatic polynomials and DP color functions of theta graphs are known.  In particular, it is well known (see e.g.,~\cite{BK23}) that when $G= \Theta(l_1, l_2,l_3)$,
\begin{align*}
&P(G,m)=\\
& \frac{((m-1)^{l_1+1}-(-1)^{l_1}(m-1))((m-1)^{l_2+1}-(-1)^{l_2}(m-1))((m-1)^{l_3+1}-(-1)^{l_3}(m-1))}{(m(m-1))^2} \\
&+ \frac{((m-1)^{l_1}+(-1)^{l_1}(m-1))((m-1)^{l_2}+(-1)^{l_2}(m-1))((m-1)^{l_3}+(-1)^{l_3}(m-1))}{m^2}.
\end{align*}
Furthermore, the following was recently shown in~\cite{BK23}.

\begin{thm} [\cite{BK23}] \label{thm: DPformulas} 
Suppose $G=\Theta(l_1,l_2,l_3)$, where $\min\{l_1,l_2,l_3\}=l_1$, $l_2 \geq 2$, and $l_3 \geq 2$.  If the parity of $l_1$ is the same as $l_3$ and different from $l_2$, then for $m \geq 2$: $$P_{DP}(G,m)=\frac{1}{m}\bigg((m-1)^{l_1+l_2+l_3}+(m-1)^{l_1}-(m-1)^{l_2}-(m-1)^{l_3+1}+(-1)^{l_2+1}(m-2)\bigg).$$
\end{thm} 
We wish to use the formula in Theorem~\ref{thm: DPformulas} along with the following recently introduced tool. 
\begin{lem} [\cite{AM25}] \label{lem: DPconnection}
Suppose $G$ is an arbitrary graph and $L$ is an $m$-assignment for $G$.  Suppose $uv \in E(G)$, $|L(u)-L(v)| = d \geq 1$, and for any $x \in L(u)$ and $y \in L(v)$ with $x \neq y$, there are at least $C$ proper $L$-colorings of $G$ that color $u$ with $x$ and $v$ with $y$.  Then,
\[P(G,L) \geq P_{DP}(G,m) + Cd.\]
\end{lem}
Before applying Theorem~\ref{thm: DPformulas} and Lemma~\ref{lem: DPconnection} we need the following lemma which will provide a bound on the constant $C$ in Lemma~\ref{lem: DPconnection} in our context of interest.
\begin{lem} \label{thm: countMath}
Suppose $G = \Theta(l_1, l_2, l_3)$ where $\min\{l_1,l_2,l_3\}=l_1$, $l_2 \geq 2$, and $l_3 \geq 2$ and $L$ is an $m$-assignment for $G$ with $m \geq 3$. Suppose $q$ and $s$ are adjacent vertices in $G$, and suppose $x \in L(q)$, $y \in L(s)$, and $x \neq y$. Then, there are at least
$$(m-1)^{l_1+l_2+l_3-5} (m-2)^2$$
proper $L$-colorings of $G$ that color $q$ with $x$ and $s$ with $y$.
\end{lem}

\begin{proof}
Clearly, $qs$ is either in the cycle contained in $G$ formed by the path of length $l_1$ and the path of length $l_2$ or in the cycle formed by the path of length $l_1$ and the path of length $l_3$.
Suppose without loss of generality that $qs$ is in the cycle, call it $C$, contained in $G$ formed by the path of length $l_1$ and the path of length $l_2$. Suppose the vertices of this cycle in cyclic order are: $q,s,v_1, \ldots, v_{l_1+l_2-2}$.  Also suppose the vertices of the path of length $l_3$ in $G$, call it $P$, in order are: $u,z_1, \ldots, z_{l_3-1},w$.

Suppose we color $q$ with $x$ and $s$ with $y$. We can now greedily complete a proper $L$-coloring of $G$ as follows. Greedily color the uncolored vertices of $C$ with a color from each vertex's list in the following order: $v_1, \ldots, v_{l_1+l_2-2}$. Notice this can be done in at least $(m-1)^{l_1+l_2-3}(m-2)$ ways. Next, color the uncolored vertices of $P$ with a color from each vertex's list in the following order: $z_1, \ldots, z_{l_3-1}$. Since this can be done in at least $(m-1)^{l_3-2}(m-2)$ ways, the result follows.
\end{proof}

We can now use Theorem~\ref{thm: DPformulas} along with Lemmas~\ref{lem: DPconnection} and~\ref{thm: countMath} to take care of most of the remaining cases.

\begin{lem} \label{lem: nochord}
Suppose $G=\Theta(l_1,l_2,l_3)$ where $\min\{l_1,l_2,l_3\} = l_1$, $l_2 \geq 2$, and $l_3 \geq 2$. Also, suppose the parity of $l_1$ is the same as $l_3$ and different from $l_2$.  If $l_1 + l_3 \geq 6$, then $G$ is enumeratively chromatic-choosable.  Moreover, if $l_1+l_3 = 4$, then $P_{\ell}(G,m) = P(G,m)$ whenever $m \geq 4$. 
\end{lem}

\begin{proof}
 Suppose $L$ is an arbitrary $m$-assignment of $G$ with $m \geq 3$ if $l_1+l_3 \geq 6$ and with $m \geq 4$ if $l_1+l_3=4$. We claim that $P(G,L) \geq P(G,m)$ which will imply that $P_{\ell}(G,m) = P(G,m)$ as desired.  If $L$ assigns the same list to every vertex of $G$, $P(G,L)=P(G,m)$. So, we may assume there is a $qs \in E(G)$ such that $L(q) \neq L(s)$. By Theorem~\ref{thm: DPformulas} and Lemmas~\ref{lem: DPconnection} and~\ref{thm: countMath}, we have 
\begin{align*} P(G,L)
&\geq \frac{1}{m}\bigg((m-1)^{l_1+l_2+l_3}+(m-1)^{l_1}-(m-1)^{l_2}-(m-1)^{l_3+1}+(-1)^{l_2+1}(m-2)\bigg) \\
&+(m-1)^{l_1+l_2+l_3-5} (m-2)^2
\end{align*}
(Note that in our application of Lemma~\ref{lem: DPconnection}, we have $d=|L(q)-L(s)|\geq 1$ since $L(q)\neq L(s)$. Thus the term $Cd$ from Lemma~\ref{lem: DPconnection} is at least $C$, and $C$ is at least $(m-1)^{l_1+l_2+l_3-5}(m-2)^2$ by Lemma~\ref{thm: countMath}.)  We also know
\begin{align*}
&P(G,m)= \\
&\frac{((m-1)^{l_1+1}-(-1)^{l_1}(m-1))((m-1)^{l_2+1}-(-1)^{l_2}(m-1))((m-1)^{l_3+1}-(-1)^{l_3}(m-1))}{(m(m-1))^2} \\
&+ \frac{((m-1)^{l_1}+(-1)^{l_1}(m-1))((m-1)^{l_2}+(-1)^{l_2}(m-1))((m-1)^{l_3}+(-1)^{l_3}(m-1))}{m^2}.
\end{align*}
Using the above inequality and equation, we calculate
$$P(G,L)-P(G,m) \geq (m-1)^{l_{1}}-(m-1)^{l_2}+(-1)^{l_2+1}(m-2)+(m-1)^{l_1+l_2+l_3-5}(m-2)^2.$$

Since $m-1 \geq 2$ when $l_1+l_3 \geq 6$ and $(m-2)^2/(m-1) \geq 4/3$ when $l_1+l_3=4$, it is easy to see that $(m-1)^{l_1}+(-1)^{l_2+1}(m-2)\geq 0$ and $(m-1)^{l_1+l_3-5}(m-2)^2 \geq 1$.  So, 
$$(m-1)^{l_1+l_2+l_3-5}(m-2)^2 - (m-1)^{l_2} = (m-1)^{l_2}((m-1)^{l_1+l_3-5}(m-2)^2-1) \geq 0.$$ Consequently, $(m-1)^{l_{1}}-(m-1)^{l_2}+(-1)^{l_2+1}(m-2)+(m-1)^{l_1+l_2+l_3-5}(m-2)^2 \geq 0$, which means $P(G,L) \geq P(G,m)$, as desired.
\end{proof}

Having proven Lemma~\ref{lem: nochord}, to complete the proof of Theorem~\ref{thm: characterize}, we need only show that: $P_{\ell}(\Theta(1,l_2,3),3)=P(\Theta(1,l_2,3),3)$ when $l_2$ is even and $P_{\ell}(\Theta(2,l_2,2),3)=P(\Theta(2,l_2,2),3)$ when $l_2$ is odd and at least 3.   

To prove these results, we make use of a lemma from~\cite{AM25} and DP-coloring.  To make the statement of the lemma easy to state, we introduce some notation that we will use from this point forward.  Whenever $G=\Theta(l_1,l_2,l_3)$ where $\min\{l_1,l_2,l_3\} = l_1$, $l_2 \geq 2$, and $l_3 \geq 2$, we will  suppose the end vertices of the paths that make up $G$ are $u$ and $v$, and we use $S_i$ to denote the path of length $l_i$ in $G$ for each $i \in [3]$.  Also, if $L$ is an $m$-assignment for $G$ and $(c,d) \in L(u) \times L(v)$, we will use $N_{i}(c,d)$ to denote the number of proper $L_i$-colorings of $S_i$ that color $u$ with $c$ and $v$ with $d$ where $L_i$ is $L$ with domain restricted to $V(S_i)$.  

\begin{lem} [\cite{AM25}] 
\label{lem: sum}
Suppose $G=\Theta(l_1,l_2,l_3)$ where $\min\{l_1,l_2,l_3\} = l_1$, $l_2 \geq 2$, and $l_3 \geq 2$.  Let $L$ be an $m$-assignment for $G$.  Then,
\[P(G,L) = \sum_{(c,d) \in L(u) \times L(v)} \prod_{i=1}^3 N_i(c,d).\]
\end{lem}

With this in mind, we now use DP-coloring to prove results on the number of guaranteed list colorings of a path.  The following result from~\cite{BK23} will play an important role.

\begin{lem} [\cite{BK23}] \label{lem: BK23}
Suppose $P$ is a path with $k$ edges where $k \in \N$ and $\mathcal{H} = (L,H)$ is a full $m$-fold cover of $P$ with $m \geq 2$.  If $x$ and $y$ are the end vertices of $P$ and there is a path in $H$ connecting $u \in L(x)$ to $v \in L(y)$, then there are 
$$\frac{(m-1)^k - (-1)^{k}}{m} + (-1)^k$$
$\mathcal{H}$-colorings of $P$ that contain $u$ and $v$.  Otherwise there are
$$\frac{(m-1)^k - (-1)^{k}}{m}$$
$\mathcal{H}$-colorings of $P$ that contain $u$ and $v$.
\end{lem}

\begin{lem}\label{lem: even path simple}
Let $P$ be a path with $k$ edges where $k \in \N$ and the vertices of $P$ written in order are
$x_0,x_1,\dots, x_k$.  Suppose $L$ is an $m$-assignment 
of $P$ with $m \ge 2$. For each $(c,d) \in L(x_0) \times L(x_k)$, let $N(c,d)$ be the number of proper $L$-colorings of $P$ where $x_0$ is 
colored with $c$ and $x_k$ is colored with $d$. Then,
$$N(c,d) \geq \min \left\{\frac{(m-1)^{k}-(-1)^{k}}{m}, \frac{(m-1)^{k}-(-1)^{k}}{m} + (-1)^{k} \right\}.$$
Moreover, there is a partition $\{A,B\}$ of $L(x_0)\times L(x_k)$ such that $|A|=m$,  $|B|=m(m-1)$, and
\[
N(c,d)\ge 
\begin{cases}
\displaystyle \frac{(m-1)^k-(-1)^k}{m} + (-1)^k & \text{if } (c,d)\in A\\[1em]
\displaystyle \frac{(m-1)^k-(-1)^k}{m} & \text{if } (c,d)\in B.
\end{cases}
\]
\end{lem}

\begin{proof}
 Suppose the cover of $P$ corresponding to $L$ is $\mathcal{H}_L = (\Lambda_L,H_L)$.  Let $H'$ be the graph obtained from $H_L$ by arbitrarily adding edges so that $E_{H'}\left(\Lambda_L(x_i),\Lambda_L(x_{i+1})\right)$ is a perfect matching for each $i \in \{0\} \cup [k-1]$.  Note that $H'$ is the disjoint union of $m$ paths. Then, let $\mathcal{H}' = (\Lambda_L,H')$.  Clearly, $\mathcal{H}'$ is a full $m$-fold cover of $P$.  Furthermore, let $\mathcal{T}$ be the set of proper $\mathcal{H}'$-colorings of $P$, and let $\mathcal{C}$ be the set of proper $L$-colorings of $P$.  Let $\mathcal{M}: \mathcal{T} \rightarrow \mathcal{C}$ be the injective function that maps each $T \in \mathcal{T}$ to $f_T \in \mathcal{C}$ where $f_T(v)$ is the second coordinate of the ordered pair in $T$ with first coordinate $v$ for each $v \in V(P)$.

Suppose $(\beta, \gamma) \in L(x_0) \times L(x_k)$.  Lemma~\ref{lem: BK23} implies that $\mathcal{T}$ has $((m-1)^{k}-(-1)^{k})/m$ elements that contain $(x_0,\beta)$ and $(x_{k},\gamma)$ if and only if there is no path in $H'$ connecting $(x_0,\beta)$ and $(x_{k},\gamma)$.  It further says that there are $(((m-1)^{k}-(-1)^{k})/m + (-1)^{k})$ elements of $\mathcal{T}$ that contain $(x_0,\beta)$ and $(x_{k},\gamma)$ if and only if there is a path in $H'$ connecting $(x_0,\beta)$ and $(x_{k},\gamma)$. 

Since $\mathcal{M}$ is injective and maps each element of $T$ containing $(x_0,\beta)$ and $(x_k, \gamma)$ to a proper $L$-coloring of $P$ that colors $x_0$ with $\beta$ and $x_k$ with $\gamma$, the desired lower bound on $N(\beta,\gamma)$ follows.  The statement concerning the partition of $L(x_0) \times L(x_k)$ follows from the fact that for exactly $m$ ordered pairs $(\eta,\mu) \in L(x_0) \times L(x_k)$, there is a path in $H'$ connecting $(x_0,\eta)$ and $(x_k, \mu)$. 
\end{proof}

Having proven Lemma~\ref{lem: even path simple}, we are now ready to prove $P_{\ell}(\Theta(1,l_2,3),3)=P(\Theta(1,l_2,3),3)$ when $l_2$ is even.  The following lemma proves something more general.

\begin{lem} \label{lem: oneedge}
Suppose $G=\Theta(1, l_2, l_3)$ where $l_2$ is even and $l_3$ is an odd integer satisfying $l_3>1$. Then, $G$ is enumeratively chromatic-choosable.
\end{lem}

\begin{proof}
Suppose $L$ is an arbitrary $m$-assignment of $G$ with $m \geq 3$.  We must show that $P(G,L) \geq P(G,m)$. By the formula for the chromatic polynomial of a theta graph, we know that 
\begin{align*}
P(G,m)= \left(\frac{(m-1)^{l_2}-1}{m}\right) \left( (m-1)^{l_3+1}+(m-1) \right).
\end{align*}
Furthermore, by \cref{lem: sum}, we know \[P(G,L) = \sum_{(c,d) \in L(u) \times L(v)} N_1(c,d)N_2(c,d)N_3(c,d).\]
By \cref{lem: even path simple} and the fact that $N_1(c,d)=0$ when $c=d$ and $N_1(c,d)=1$ when $c \neq d$, we have
$$P(G,L) = \sum_{(c,d) \in L(u) \times L(v), c \neq d} N_1(c,d)N_2(c,d)N_3(c,d)
\geq \frac{(m-1)^{l_2}-1}{m} \sum_{(c,d) \in L(u) \times L(v), c \neq d} N_3(c,d).$$
Now, suppose the graph $M$ is the cycle obtained from the path $S_3$ by adding an edge between $u$ and $v$. Since each proper $L_3$-coloring of $M$ corresponds to a proper $L_3$-coloring of $S_3$ that colors $u$ and $v$ differently, $P(M,L_3)= \sum_{(c,d) \in L(u) \times L(v), c \neq d} N_3(c,d)$. Since cycles are enumeratively chromatic-choosable by Theorem~\ref{thm: examples} and $M$ is a $(l_3+1)$-cycle, $P(M,L_3) \geq P(C_{l_3+1},m)=(m-1)^{l_3+1}+(m-1).$ Consequently,

$$\frac{(m-1)^{l_2}-1}{m} \sum_{(c,d) \in L(u) \times L(v), c \neq d} N_3(c,d) \geq \left(\frac{(m-1)^{l_2}-1}{m}\right) \left( (m-1)^{l_3+1}+(m-1) \right)$$
as desired.
\end{proof}

Finally, we complete the proof of Theorem~\ref{thm: characterize} by proving $P_{\ell}(\Theta(2,l_2,2),3)=P(\Theta(2,l_2,2),3)$ when $l_2$ is odd and at least 3. 

\begin{lem} \label{lem: finish}
Suppose $G=\Theta(2,l_2,2)$ where $l_2$ is odd and $l_2 \geq 3$.  Then, $P_{\ell}(G,3)=P(G,3).$    
\end{lem}

\begin{proof}
Suppose $L$ is an arbitrary $3$-assignment of $G$.  We must show that $P(G,L) \geq P(G,3)$. By the formula for the chromatic polynomial of a theta graph, we know that 
\begin{align*}
P(G,3)= 2^{l_2+1}+2 + (2)^2(2^{l_2}-2) = 6(2^{l_2}-1).
\end{align*}
Furthermore, by \cref{lem: sum}, we know \[P(G,L) = \sum_{(c,d) \in L(u) \times L(v)} N_1(c,d)N_2(c,d)N_3(c,d).\]
We will first show the desired result in the case that $L(u)=L(v)$ and then in the case that $L(u) \neq L(v)$.  Suppose $L(u)=L(v)$. Notice that for each $i \in \{1,3\}$, we have that $N_i(c,d) \geq 1$ when $c \neq d$ and $N_i(c,d) \geq 2$ otherwise.  Consequently, 
\begin{align*}
P(G,L) &= \sum_{(c,d) \in L(u) \times L(v), c \neq d} N_1(c,d)N_2(c,d)N_3(c,d) + \sum_{(c,d) \in L(u) \times L(v), c = d} N_1(c,d)N_2(c,d)N_3(c,d) \\
&\geq \sum_{(c,d) \in L(u) \times L(v), c \neq d} N_2(c,d) + 4\sum_{(c,d) \in L(u) \times L(v), c = d} N_2(c,d).    
\end{align*}
Now, suppose the graph $M$ is the cycle obtained from the path $S_2$ by adding an edge between $u$ and $v$. Notice $P(M,L_2)= \sum_{(c,d) \in L(u) \times L(v), c \neq d} N_2(c,d)$. Since cycles are enumeratively chromatic-choosable by Theorem~\ref{thm: examples} and $M$ is a $(l_2+1)$-cycle, $P(M,L_2) \geq P(C_{l_2+1},3)=2^{l_2+1}+2.$  Now, using this fact, \cref{lem: even path simple}, and the fact that there are exactly three pairs in $L(u) \times L(v)$ that have the same first and second coordinate, we obtain 
\begin{align*}
P(G,L) &\geq \sum_{(c,d) \in L(u) \times L(v), c \neq d} N_2(c,d) + 4\sum_{(c,d) \in L(u) \times L(v), c = d} N_2(c,d) \\
&\geq 2^{l_2+1}+2 + (3)(4)\left( \frac{2^{l_2}+1}{3} - 1\right) = P(G,3).
\end{align*}

Now we turn our attention to the case where $L(u) \neq L(v)$.  We begin by giving a lower bound on the number of pairs $(c,d) \in L(u) \times L(v)$ with the property $N_i(c,d) \geq 2$ where $i \in \{1,3\}$.  With this in mind, suppose that $w$ is the internal vertex of $S_1$.

Let $\beta = |L(w) \cap L(u) \cap L(v)|$, $\gamma = |L(w) \cap (L(u)-L(v))|$, and $\mu = |L(w) \cap (L(v)-L(u))|$.  Notice that $\beta \in \{0,1,2\}$ and $\beta+\gamma+\mu \leq |L(w)|=3$.  We have that $N_1(c,d) \geq 1$ for each $(c,d) \in L(u) \times L(v)$.  Furthermore in order for $N_1(c,d)=1$, it must be that $c \neq d$ and $c,d \in L(w)$.  Consequently, the number of elements $(c,d) \in L(u) \times L(v)$ satisfying $N_1(c,d)=1$ is
$$(\gamma+\beta)(\mu+\beta)-\beta.$$
By using the facts that $\beta$, $\gamma$, and $\mu$ are nonnegative integers, $\beta \in \{0,1,2\}$, and $\beta+\gamma+\mu \leq 3$, one can easily verify that $(\gamma+\beta)(\mu+\beta)-\beta \leq 4$.  Consequently, there are at least 5 ordered pairs $(c,d) \in L(u) \times L(v)$ satisfying $N_1(c,d) \geq 2$.  Similarly, there are at least 5 (possibly different) ordered pairs $(c,d) \in L(u) \times L(v)$ satisfying $N_3(c,d) \geq 2$. Furthermore, by Lemma~\ref{lem: even path simple}, we know there is a partition $\{A,B\}$ of $L(u)\times L(v)$ such that $|A|=3$,  $|B|=6$, and
\[
N_2(c,d)\ge 
\begin{cases}
\displaystyle \frac{2^{l_2}-2}{3} & \text{if } (c,d)\in A\\[1em]
\displaystyle \frac{2^{l_2}+1}{3} & \text{if } (c,d)\in B.
\end{cases}
\]
Using these facts along with the AM-GM inequality, we see
\begin{align*}
P(G,L) &= \sum_{(c,d) \in L(u) \times L(v)} N_1(c,d)N_2(c,d)N_3(c,d) \\
&\geq 9 \left( \prod_{(c,d)\in L(u) \times L(v)} N_1(c,d)N_2(c,d)N_3(c,d)\right)^{1/9} \\ 
&\geq   9 \left( (2^5)(2^5)\left(\frac{2^{l_2}-2}{3} \right)^3\left(\frac{2^{l_2}+1}{3} \right)^6\right)^{1/9}  \\
&\geq 6 ((2^{l_2}-2)(2^{l_2}+1)^2)^{1/3}
\end{align*}
Since $l_2 \geq 3$, it is easy to verify that $(2^{l_2}-2)(2^{l_2}+1)^2 \geq (2^{l_2}-1)^3$.  Consequently, $P(G,L) \geq 6(2^{l_2}-1)=P(G,3)$ as desired. 
\end{proof}

{\bf Acknowledgment.} The authors would like to thank Sarah Allred and Hemanshu Kaul for helpful conversations.  The authors would also like to thank the anonymous referee for helpful comments.  This project was completed by the Coloring Research Group of South Alabama at the University of South Alabama during the fall 2025 and spring 2026 semesters.  The support of the University of South Alabama is gratefully acknowledged.

\bibliographystyle{hplain}
\bibliography{bibliography}{}
\vspace{-5mm}

\end{document}